\documentclass[10pt]{article}
\usepackage{lmodern}
\usepackage[T1]{fontenc}
\usepackage{amsmath,amssymb,amsfonts,amsopn,amsthm,amsbsy}
\usepackage{latexsym}
\usepackage{color}
\usepackage{cite}  
\usepackage{xfrac}
\usepackage{bbm}   

\DeclareMathOperator{\sym}{sym}
\DeclareMathOperator{\skyw}{skew}
\DeclareMathOperator{\curl}{curl}
\DeclareMathOperator{\Curl}{Curl}
\DeclareMathOperator{\dyw}{div}
\DeclareMathOperator{\Dyw}{Div}

\DeclareMathOperator{\tr}{tr}

\newcommand{\nn}{\nonumber}

\newcommand{\R}{{\mathbb R}}

\newcommand{\di}{{\mathrm d}}

\newcommand{\st}{\mu_{\mathrm{micro}}}

\newcommand{\bigchi}{\raisebox{3pt}{$\chi$}}

\newcommand{\id}{\mathbbm{1}}

\renewcommand{\leq}{\leqslant}
\renewcommand{\geq}{\geqslant}
\makeatletter
\let\@fnsymbol\@arabic
\makeatother
\title{{\bf Existence results  for non-homogeneous boundary conditions in the  relaxed micromorphic model}}
\author{Ionel-Dumitrel Ghiba\footnote{Ionel-Dumitrel Ghiba, \   Alexandru Ioan Cuza University of Ia\c si, Department of Mathematics,  Blvd. Carol I, no. 11, 700506 Ia\c si,
		Romania;  Octav Mayer Institute of Mathematics of the
		Romanian Academy, Ia\c si Branch,  700505 Ia\c si; email: dumitrel.ghiba@uaic.ro}	\quad
	and \quad
	Patrizio Neff\,\thanks{Patrizio Neff, \ \  Head of Lehrstuhl f\"{u}r Nichtlineare Analysis und Modellierung, Fakult\"{a}t f\"{u}r Mathematik, Universit\"{a}t Duisburg-Essen, Campus Essen, Thea-Leymann Str. 9, 45127 Essen, Germany, email: patrizio.neff@uni-due.de}\quad
	and \quad Sebastian Owczarek\,\thanks{Sebastian Owczarek, \ \  Faculty of Mathematics and Information Science, Warsaw University of Technology, ul. Koszykowa 75, 00-662 Warsaw, Poland; email: s.owczarek@mini.pw.edu.pl}	}
\date{}

\renewcommand{\theequation}{\thesection.\arabic{equation}}
\newtheorem{tw}{Theorem}[section]

\newtheorem{de}[tw]{Definition}
\newtheorem{col}[tw]{Corollary}
\newtheorem{uwa}[tw]{Remark}

\begin{document}
\maketitle
\bigskip
\begin{abstract}
	
In this paper we use a  property of the extension operator from the space of tangential traces of ${\rm H}({\rm curl};\Omega)$ in the  context of the linear  relaxed micromorphic model, a theory which is recently used to describe the behaviour of some metamaterials showing unorthodox behaviors with respect to elastic wave
propagation.  We show that the new property is important for existence  results of strong solution for non-homogeneous boundary condition in both the dynamic  and the static case.

\bigskip

\smallskip

\end{abstract}
\tableofcontents
\section{Introduction}

 In some recent papers \cite{GhibaNeffExistence,NeffGhibaMadeoLazar}, the last two authors and their collaborators have studied the existence of the solution of certain initial-boundary value problems in the relaxed micromorphic models \cite{NeffGhibaMicroModel} under homogeneous boundary conditions. These models are constructed as useful alternatives to the classical micromorphic theory \cite{Eringen64,Mindlin64,Eringen99}, which is difficult to be used for practical problems, first of all due to the very large number of constitutive parameters which have to be determined from experiments.  The micromorphic theory \cite{Eringen64,Mindlin64} is a generalised theory of continua which is capable   to describe both macro- and micro-deformation  by  considering that any point of the body is endowed with two   {fields}, a macroscopic vector field $u:\Omega\times[0,T]\rightarrow\mathbb{R}^3$ for the displacement of the macroscopic material points, and a macroscopic tensor field 
$P:\Omega\times[0,T]\rightarrow\mathbb{R}^{3\times3}$ describing the micro-deformation (micro-distortion) of  the substructure of  the material. In this paper we consider $\Omega$ to be a connected, bounded and open subset of $\mathbb{R}^3$ with a  ${\rm C}^{1,1}$  boundary $\partial\Omega$ and $T > 0$ denotes a fixed  time interval. In classical elasticity it is assumed that the body is composed by material points. However, the material points   are  physical representations of some domains for each of which the shapes and the rotations are ignored. In the micromorphic theory, each such ``small'' domain actually   may rotate, stretch, shear and shrink (the micro-distortion). All these informations are stored by the additional tensor $P\in \mathbb{R}^{3\times 3}$, called the  micro-distortion tensor. In the relaxed micromorphic model we use $\Curl P\in \mathbb{R}^{3\times 3}$ as constitutive variable, instead of $\nabla P\in \mathbb{R}^{3\times 3\times 3}$ used in the classical micromorphic theory. This choice has allowed us to reduce the number of constitutive coefficients involved in the model.  In contrast to the classical micromorphic theory, considering a wave ansatz in the partial differential equations arising in the relaxed micromorphic model, some band gaps occur in the diagrams describing the dispersion curves, see \cite{MadeoNeffGhibaW,madeo2016reflection,d2017panorama,madeo2017review,d2017effective,NeffMadeoEidel2019JELast,aivaliotis2019microstructure,barbagallo2019relaxed}. This means that there exists an interval of wave frequencies for which no wave propagation may  occur \cite{MadeoNeffGhibaW,madeo2016complete,madeo2016reflection,madeo2019dispersion}. Such an interesting phenomena is  obtained in experiments when meta-materials are designed and cannot be captured by the classical micromorphic approach \cite{d2017panorama,madeo2017review}. These meta-materials are able to ``absorb'' or even ``bend'' elastic waves with no energetic cost (see e.g. \cite{blanco2000large,liu2000locally}).

Another interesting aspect is that, in the relaxed micromorphic theory,  we were able to prove the existence of the solution of the related initial-boundary value problem with homogeneous boundary conditions  without assuming that the corresponding energy is positive definite in terms of the constitutive variables used in the classical micromorphic model. In the dynamic case and for homogeneous boundary conditions,  the strong solution  $(u,P)$ belongs to ${\rm C}^1([0,T)),{\rm H}_0^1(\Omega))\times {\rm C}^1([0,T)),{\rm H}_0(\Curl;\Omega))$, while in the static case the weak solution belongs to ${\rm H}_0^1(\Omega)\times {\rm H}_0(\Curl;\Omega)$.

However, in contrast to other mechanical models, there are no physical evidences why the boundary conditions (Dirichlet boundary conditions on the displacement $u$ and tangential boundary conditions on the micro-distortion tensor $P$) must be homogeneous.  In order to obtain  existence and uniqueness results  for non-homogeneous boundary conditions, the idea is the same as for classical linear elasticity \cite{valent2013boundary} (see also \cite{hughes1978classical}), i.e. we have to remove the boundary conditions. There are no additional difficulties in order to remove the Dirichlet boundary conditions $u=g$ on $\partial \Omega$ as long as $g\in {\rm H}^{\frac{1}{2}}(\partial\Omega)$, since we use the standard extensions operator from the trace space ${\rm H}^{\frac{1}{2}}(\partial\Omega)$ to ${\rm H}^{1}(\Omega)$ and everything is reduced to the situation when we have homogeneous Dirichlet boundary conditions. However, we do not expect the same when we intend  to remove the tangential boundary condition on the micro-distortion tensor $P_i\times n=G_i$ on $\partial \Omega$, $i=1,2,3$, where $P_i$ are the lines of the tensor $P$. Indeed, using the results established by Alonso and Valli \cite{tantrace}, the functions $G_i$ have to belong at least to the space $
\bigchi_{\partial\Omega}:=\{v\in {\rm H}^{-\frac{1}{2}}(\partial\Omega)\,  :\,\langle v, n\rangle=0\,\,\mathrm{ on }\ \partial \Omega \ \mathrm{and}\,\,\dyw_{\tau} v\in {\rm H}^{-\frac{1}{2}}(\partial\Omega)\}
$, where $\dyw_{\tau} v$ is the tangential divergence of $v\in {\rm H}^{-\frac{1}{2}}(\partial\Omega)$, in order to have a clear meaning of the boundary condition $P_i\times n=G_i$ on $\partial \Omega$  in the sense of ${\rm H}^{-\frac{1}{2}}(\partial\Omega)$. In the framework of the results presented in  \cite{tantrace},  the first idea is to use directly  the extension operator from the space  $\bigchi_{\partial\Omega}$ to ${\rm H}(\Curl;\Omega)$ since we know that for each $G\in \bigchi_{\partial\Omega}$ we have that $\Curl \widetilde{G}\in {\rm L}^2(\Omega)$, where $\widetilde{G}$ is the extension of $G$ in ${\rm H}(\Curl;\Omega)$. But, in order to obtain the existence result in the relaxed micromorphic model (see Section 3), the extension $\widetilde{G}$ of $G$ (which prescribes the values of $P_i\times n$ on the boundary)  has  to be such that $\Curl\Curl \widetilde{G}\in {\rm L}^2(\Omega)$ (or   $\Curl \widetilde{G}\in ({\rm H}_0(\Curl;\Omega))^*$ if we are interested in  a weak solution)  and ${\rm Div} (\sym P), {\rm Div}({\rm skew} P)\in {\rm L}^2(\Omega)$, bioth of which are unknown  from the construction given by Alonso and Valli \cite{tantrace}.     However, the extension operator constructed by Alonso and Valli \cite{tantrace} is such that ${\rm Curl}({\rm Curl}\,\widetilde{G})=0$. If we assume that $G\in \widetilde{\chi}_{\partial\Omega}:= \{G\in {\rm H}^{\frac{1}{2}}(\partial\Omega)\,\,:\,\, \langle G_i, n\rangle=0\  \mathrm{ on } \ \partial \Omega \}$, then from \cite[Theorem 6 of Section 2]{electrobook} we deduce that the extension $\widetilde{G}$ belongs  to ${\rm H}^1(\Omega)$. In Section 3, with the help of this property,  we are able to show an  existence result for the initial-boundary value problem with non-homogeneous boundary condition. We have no information about ${\rm Div} \,P$ in $\Omega$ and even about $P_i\cdot n$, $i=1,2,3$, on $\partial \Omega$ in the considered model. In the last section of the paper, we  show that a similar technique leads  to the existence of the weak solution for the static problem with non-homogeneous boundary condition. In a future contribution we will use these results in order to deal with a local higher regularity result for the micro-distortion tensor $P$ and for the displacement $u$.
 	
\section{The extension operator on the space of tangential traces of ${\rm H}(\curl;\Omega)$}

As usual, the space ${\rm H}^1(\Omega)$ denotes the Hilbert space of vectors from ${\rm L}^2(\Omega)$ such that all  weak partial derivatives of order one belong also to ${\rm L}^2(\Omega)$. The space ${\rm H}^{\frac{1}{2}}(\partial\Omega)$ represents the trace space of ${\rm H}^1(\Omega)$ over $\partial \Omega$, while the space ${\rm H}^1_0(\Omega)$ is the subspace of functions $v\in {\rm H}^1(\Omega)$ which are zero on the boundary $\partial\Omega$ in the sense of ${\rm H}^{\frac{1}{2}}(\partial\Omega)$. 

In order to prove  existence and uniqueness results for  the relaxed micromorphic model  with non-homogeneous boundary conditions we are going to remove the tangential conditions.
Now, let us introduce the space ${\rm H}(\curl;\Omega)$
\begin{equation}
{\rm H}(\curl;\Omega)=\{v \in {\rm L}^2(\Omega)\mid \curl v\in {\rm L}^2(\Omega)\}\,,\quad \curl=\nabla\times\,.
\label{2.1}
\end{equation}
The space ${\rm H}(\curl;\Omega)$ is a Hilbert space with norm
\begin{equation}
\|v\|^2_{{\rm H}(\curl;\Omega)}=\|v\|^2_{{\rm L}^2(\Omega)}+\|\curl v\|^2_{{\rm L}^2(\Omega)}\,.
\label{2.2}
\end{equation}
 By ${\rm H}_0(\curl;\Omega)$ we will denote the closed subspace of vectors $v\in {\rm H}(\curl;\Omega)$ such that $v\times n=0$ on the boundary $\partial\Omega$ in the sense of ${\rm H}^{-\frac{1}{2}}(\partial\Omega)$  (the dual space of ${\rm H}^{\frac{1}{2}}(\partial\Omega)$).\\
Moreover, we consider the set
\begin{equation}
{\rm H}(\dyw;\Omega)=\{v \in {\rm L}^2(\Omega)\mid \dyw v\in {\rm L}^2(\Omega)\}\,,
\label{hdyw}
\end{equation}
which is a Hilbert space with norm
\begin{equation}
\|v\|^2_{{\rm H}(\dyw;\Omega)}=\|v\|^2_{{\rm L}^2(\Omega)}+\|\dyw v\|^2_{{\rm L}^2(\Omega)}\,.
\label{hdywnorm}
\end{equation}
Similarly, ${\rm H}_0(\dyw;\Omega)$ is the subspace of vectors $v\in {\rm H}(\dyw;\Omega)$ satisfying $\langle v, n\rangle=0$ on  $\partial\Omega$ in the sense of ${\rm H}^{-\frac{1}{2}}(\partial\Omega)$.

We recall the definition of the tangential divergence of  vectors from the space ${\rm H}^{\frac{1}{2}}(\partial\Omega)$ (for more details we refer to \cite{tandyw}).
\begin{de}
	Let us assume that $v\in {\rm H}^{-\frac{1}{2}}(\partial\Omega)$ be a vector field satisfying $\langle v, n\rangle=0$ on $\partial \Omega$. The tangential divergence $\dyw_{\tau} v$ of the vector $v$ is the distribution in ${\rm H}^{-\frac{3}{2}}(\partial\Omega)$ which satisfies
	
	\begin{equation}
	[[\dyw_{\tau} v,w]]_{\partial\Omega}=- [v,(\nabla w^{\ast})\big|_{\partial\Omega}]_{\partial\Omega}\quad \forall\   w\in {\rm H}^{\frac{3}{2}}(\partial\Omega)\,,
	\label{2.3}
	\end{equation}
	where $w^{\ast}\in {\rm H}^{2}(\Omega)$ is any extension of $w$ in $\Omega$. Here, $[\cdot,\cdot]_{\partial\Omega}$ denotes the duality pair between the space ${\rm H}^{-\frac{1}{2}}(\partial\Omega)$ and ${\rm H}^{\frac{1}{2}}(\partial\Omega)$.  $[[\cdot,\cdot]]_{\partial\Omega}$ denotes the duality pair between the space ${\rm H}^{-\frac{3}{2}}(\partial\Omega)$ and ${\rm H}^{\frac{3}{2}}(\partial\Omega)$.
\end{de}

It is well known that $(n\times v)\big|_{\partial \Omega}$ belongs to ${\rm H}^{-\frac{1}{2}}(\partial\Omega)$ provided that $v\in {\rm H}(\curl;\Omega)$ (see \cite[p. 34]{Giraultbook}). In fact, if the boundary satisfies some regularity conditions, the tangential trace belongs to a proper subspace of 
${\rm H}^{-\frac{1}{2}}(\partial\Omega)$ defined  by 
\begin{equation}
\bigchi_{\partial\Omega}:=\{v\in {\rm H}^{-\frac{1}{2}}(\partial\Omega) : \langle v, n\rangle=0\,\,\ \mathrm{ on }\  \partial \Omega \ \mathrm{and}\,\,\dyw_{\tau} v\in {\rm H}^{-\frac{1}{2}}(\partial\Omega)\}
\label{2.4}
\end{equation}
equipped with the norm 
\begin{equation}
\|v\|_{\bigchi_{\partial\Omega}}=\|v\|_{{\rm H}^{-\frac{1}{2}}(\partial\Omega)}+\|\dyw_{\tau} v\|_{{\rm H}^{-\frac{1}{2}}(\partial\Omega)}\,.
\label{2.5}
\end{equation}
This result was obtained by Alonso and Valli \cite{tantrace}.
\begin{tw}{\rm \cite[p.~165]{tantrace}}
	\label{lem:2.2}
	Assume that the boundary $\partial\Omega$ is of class ${\rm C}^{1,1}$  or that $\Omega$ is a convex polyhedron. Then  for any vector $v\in {\rm H}(\curl;\Omega)$ the tangential trace $(n\times v)\big|_{\partial\Omega}$ belongs to $\bigchi_{\partial\Omega}$ and there exists a linear and continuous extension operator from the space $\bigchi_{\partial\Omega}$ to ${\rm H}(\curl;\Omega)$.\hfill $\blacksquare$
\end{tw}

Therefore,  for any vector $v\in {\rm H}(\curl;\Omega)$ there exists the constant $C_{\Omega}>0$ (independent of $v$), such that the  inequality 
\begin{equation}
\|v\|_{\bigchi_{\partial\Omega}}\leq C_{\Omega}\,\|v\|_{{\rm H}(\curl;\Omega)}
\label{2.6}
\end{equation}
holds.
Let us denote by $\gamma_{\tau}:{\rm H}(\curl;\Omega)\rightarrow \bigchi_{\partial\Omega}$ the tangential trace mapping, which is continuous. In  \cite{tantracena1,tantracena2,tantracena3} it is proved that the mapping $\gamma_{\tau}$ is also surjective. 

 For better understanding and since it is very useful to visualize a specific step of the construction proposed by Alonso and Valli \cite{tantrace}, we  shortly present the construction of the extension operator.\\[1ex]
$1.$\hspace{1ex} Assume that $v\in\bigchi_{\partial\Omega}$ and first let us consider the following Neumann problem
\begin{equation}
\begin{split}
-\Delta w(v)&=0\qquad\qquad \mathrm{in\,\,} \quad  \Omega\,,\\
\frac{\partial w(v)}{\partial n}&= -\dyw_{\tau} v\quad \mathrm{on\,\,}\quad \partial\Omega\,,\\
\int_{\Omega} w(v)\,\di x&=0 \qquad\qquad \mathrm{in\,\,} \quad  \Omega\,.
\end{split}
\label{2.9}
\end{equation}
Lemma 3.2 in \cite{tantrace} implies that there exists a unique solution $ w(v)\in H^1(\Omega)$ of \eqref{2.9}, which satisfies the estimate
\begin{equation}
\begin{split}
\|w(v)\|_{{\rm H}^1(\Omega)}&\leq C^{\ast}_1\,\|\dyw_{\tau} v\|_{{\rm H}^{-\frac{1}{2}}(\partial\Omega)}\,,
\end{split}
\label{2.19}
\end{equation}
where $C^{\ast}_1>0$ only depends on $\Omega$.\\[1ex]
$2.$\hspace{1ex} Let us introduce the set of Neumann harmonic vectorfields
\begin{equation}
\begin{split}
\mathcal{H}(n):=\{\lambda\in {\rm L}^2(\Omega)\mid \curl \lambda=0,\,\, \dyw\lambda=0,\,\, \langle\lambda, n\rangle\big|_{\partial\Omega}=0\}\,.
\end{split}
\label{2.20}
\end{equation}
It is know that this vector space has a finite dimension $n$ (see for instant \cite{temharmonic,picardharmonic}). Let us set $\{\lambda_i\}_{i=1}^{n}$ an orthonormal basis of $\mathcal{H}(n)$ equipped with the scalar product in ${\rm L}^2(\Omega)$. 

Moreover, define the function 
$
\mu(v)=\sum_{i=1}^{n} [v, \lambda_{i}\big|_{\partial\Omega} ]_{\partial\Omega}\,\lambda_i\,.
$
Finally, set 
$
W=\{\phi\in {\rm H}(\curl;\Omega)\cap {\rm H}_0(\dyw;\Omega)\mid \int_{\Omega}\phi\,\lambda\,\di x=0,\,\,\forall\,\,\lambda\in\mathcal{H}(n)\}
$
and consider the second auxiliary problem: we are looking for a function $r(v)\in W$ such that 
\begin{equation}
\begin{split}
\int_{\Omega}\big(\curl r(v)\curl\phi&+\dyw r(v)\dyw\phi\big)\,\di x=\int_{\Omega}\big(\nabla w(v)\,\phi+\mu(v)\phi\big)\,\di x-[ v, \phi\big|_{\partial\Omega}]_{\partial\Omega}\quad\mathrm{for\, all}\,\, \phi\in W\,.
\end{split}
\label{2.24}
\end{equation}
 Observe that the function $\phi\in W$ and $\lambda_i$ are well defined on $\partial\Omega$ and both belong to ${\rm H}^{\frac{1}{2}}(\Omega)$. This follows from the  property
\begin{equation}
\begin{split}
{\rm H}(\curl;\Omega)\cap {\rm H}_0(\dyw;\Omega)\hookrightarrow {\rm H}^1(\Omega)\,,
\end{split}
\label{2.25}
\end{equation}
which is satisfied for a ${\rm C}^{1,1}$ boundary of $\Omega$ (the general reference is \cite{Giraultbook}).

\noindent
Lemma 3.4 of \cite{tantrace} yields that there exists a unique solution $r(v)\in W$ of  problem \eqref{2.24} which satisfies 
\begin{equation}
\begin{split}
\|\curl r(v)\|_{{\rm L}^2(\Omega)}\leq C_1\big(\|\nabla w(v)\|_{{\rm L}^2(\Omega)}+\|v\|_{{\rm H}^{-\frac{1}{2}}(\partial\Omega)} \big)
\end{split}
\label{2.26}
\end{equation}
for all $v\in\bigchi_{\partial\Omega}$, where the positive constant $C_1$ depends on $\Omega$, only. In addition, the problem \eqref{2.24} is satisfied in a ``strong'' sense (for almost all $x\in\Omega$) and $\dyw r(v)=0$ (see Lemma 3.5 of \cite{tantrace}).\vspace{2mm}\\
$3.$\hspace{1ex} Now we are able to construct an extension operator. Let us define $\mathcal{R}_{\partial\Omega}(v):=\curl r(v)$. We observe that:\\[1ex] 
\textbullet\hspace{1ex} $\curl r(v)\in {\rm L}^2(\Omega)$ (by inequality \eqref{2.26})\,,\\[1ex]
\textbullet\hspace{1ex} $\curl\curl r(v)=(\nabla w(v)+\mu(v))\in {\rm L}^2(\Omega)$\,,\\[1ex]
\textbullet\hspace{1ex} $\mathcal{R}_{\partial\Omega}:\bigchi_{\partial\Omega}\rightarrow {\rm H}(\curl;\Omega)$\,,\\[1ex]
\textbullet\hspace{1ex} $(n\times\mathcal{R}_{\partial\Omega}v)_{\partial\Omega}=v$ on $\partial\Omega$ (\eqref{2.24} is satisfied in the strong sense).\\[1ex]
The construction of the extension operator is finished. 

\begin{col}\label{corn} For all $v\in\bigchi_{\partial\Omega}$ the following is true:
	 $\curl\curl\mathcal{R}_{\partial\Omega}(v)=0\,;$
 $\dyw\,( \mathcal{R}_{\partial\Omega}(v))=0$\,;
 $\mathcal{R}_{\partial\Omega}(v)\in {\rm H}^1(\Omega)$ for all $v\in \widetilde{\chi}_{\partial\Omega}:= \{u\in {\rm H}^{\frac{1}{2}}(\partial\Omega)\,\,;\,\, \langle u, n\rangle=0\}$.
\end{col}
\begin{proof}
	 From the 3rd step of the construction of the extension operator, we have $\curl\curl\mathcal{R}_{\partial\Omega}(v)=\curl\nabla w(v)+\curl\mu(v)=\curl\nabla w(v)=0\,.$
		The proof of the second property of $\mathcal{R}_{\partial\Omega}(v)$  is immediately.

 Additionally, if we assume that $v\in \widetilde{\chi}_{\partial\Omega}$, then we may use \cite[Theorem 6 of Section 2]{electrobook}: Let  $\Omega $ be a  regular (${\rm C}^{1,1}$) open set in $\mathbb{R}^3$ and let $u\in {\rm L}^2(\Omega)$ with $\curl u\in {\rm L}^2(\Omega), {\rm div} u\in {\rm L}^2(\Omega)$, and $\langle u, n\rangle\in {\rm H}^{\frac{1}{2}}(\partial\Omega)$ (resp. $u\times n\in {\rm H}_t^{\frac{1}{2}}(\partial\Omega)$, where ${\rm H}_t^{\frac{1}{2}}(\partial\Omega)$ is the closure in ${\rm H}^{-\frac{1}{2}}(\partial\Omega)$ of ${\rm L}^2_t(\partial \Omega):=\{u\in {\rm L}^2(\Omega), \ u\times n=0\ \mathrm{ on }\  \partial \Omega\}$). Then, it follows that $u\in {\rm H}^1(\Omega)$.

	Hence, using the above result we deduce $\mathcal{R}_{\partial\Omega}(v)\in {\rm H}^1(\Omega)$.
\end{proof}

\section{The relaxed micromorphic model. Preliminary resuls}
\renewcommand{\theequation}{\thesection.\arabic{equation}}
\setcounter{equation}{0}%

We consider now that a relaxed micromorphic continuum  occupies the domain $\Omega$ and that   the motion of the body is referred to a fixed system of rectangular Cartesian axes $Ox_i$, $(i=1,2,3)$. Throughout this paper {(if we do not specify otherwise)} Latin subscripts take the values $1,2,3$. 
We denote by $\mathbb{R}^{3\times 3}$ the set of real $3\times 3$ matrices.   For all $X\in\mathbb{R}^{3\times3}$ we set ${\rm sym}\, X=\frac{1}{2}(X^T+X)$ and ${\rm skew} X=\frac{1}{2}(X-X^T)$.
The standard Euclidean scalar product on $\mathbb{R}^{3\times 3}$ is given by
$\langle {X},{Y}\rangle_{\mathbb{R}^{3\times3}}=\tr({X Y^T})$, and thus the Frobenius tensor norm is
$\|{X}\|^2=\langle{X},{X}\rangle_{\mathbb{R}^{3\times3}}$. In the following we omit the index
$\mathbb{R}^{3\times3}$. The identity tensor on $\mathbb{R}^{3\times3}$ will be denoted by $\id$, so that
$\tr({X})=\langle{X},{\id}\rangle$. Typical conventions for differential
operations are implied such as comma followed
by a subscript to denote the partial derivative with respect to
the corresponding cartesian coordinate, while $t$ after a comma denotes the partial derivative with respect to the time.

In the rest of the paper,  for vector fields $v$ with components in ${\rm H}^{1}(\Omega)$ 
we consider
$
\nabla\,v:=\big(
({\rm grad}\,  v_1)^T,
({\rm grad}\, v_2)^T,$ $
({\rm grad}\, v_3)^T
\big)^T\in \mathbb{R}^{3\times 3}.
$ For  tensor fields $P$ with rows in ${\rm H}({\rm curl}\,; \Omega)$, i.e.
$
	P=\left(
	P_1^T,    P_2^T,    P_3^T
	\right)^T\in \mathbb{R}^{3\times 3},\, \quad P_i\in {\rm H}({\rm curl}\,; \Omega)\,,
$
	we define
$
	  {\rm Curl}\,P:=\left(
	({\rm curl}\, P_1)^T,
	({\rm curl}\,P_2)^T,
	({\rm curl}\,P_3)^T
	\right)^T\in \mathbb{R}^{3\times 3}.
$
 The corresponding Sobolev spaces for the second order tensor fields $P$ and ${\rm Curl}\, P$   will be denoted by
$ {\rm H}({\rm Curl}\,; \Omega)\, 
$ and $ {\rm H}_0({\rm Curl}\,; \Omega)\, ,
$ respectively.

The partial differential equations in the unknown functions $u:\Omega\times [0,T]\rightarrow \R^3$ and $P:\Omega\times [0,T]\rightarrow \R^{3\times 3}$ associated to the dynamical relaxed micromorphic model \cite{NeffGhibaMicroModel} are
\begin{eqnarray}
\label{1.2}
u_{,tt}&=&\Dyw\big(2\mu_{\rm e}\sym(\nabla u-P)+2\mu_{\rm c}\skyw(\nabla u-P)+\lambda_{\rm e}\tr(\nabla u-P)\id\big)+F\nn\\[1ex]
P_{,tt}&=&2\mu_{\rm e}\sym(\nabla u-P)+2\mu_{\rm c}\skyw(\nabla u-P)+\lambda_{\rm e}\tr(\nabla u-P)\id\nn\\[1ex]
&&-(2\st\sym P+\lambda_{\mathrm{micro}}(\tr P)\id)-\mu_{\rm macro} L_{\rm c}^2\Curl\Curl P+M\,,
\end{eqnarray}
in $\Omega\times (0,T)$, where $F:\Omega\times (0,T)\rightarrow \R^3$ is a given body force and $M:\Omega\times (0,T)\rightarrow \R^{3\times 3}$ is a given body moment tensor.

Here, the constants  $\mu_{\rm e},\lambda_{\rm e},\mu_{\rm c}, \mu_{\rm micro}, \lambda_{\rm micro}$  are constitutive parameters describing the elastic response of the material, while $L_{\rm c}>0$ is the characteristic length of the relaxed micromorphic
model. The limit case $L_{\rm c}\to 0$ corresponds to considering very large specimens of a microstructured meta-material \cite{barbagallo2017transparent}. We suppose that the constitutive parameters are such that 
\begin{align}\label{condpara}
\mu_{\rm e}>0,\quad\quad  2\mu_{\rm e}+3\lambda_{\rm e}>0,\quad\quad  \mu_{\rm c}\geq 0,\quad\quad  \mu_{\rm micro}>0, \quad\quad  2\mu_{\rm micro}+3\lambda_{\rm micro}>0, \qquad \mu_{\rm macro}>0.
\end{align}
These assumptions assure that the corresponding potential energy
\begin{align}
I(\nabla u,P)=\int_{\Omega}\Big(&\mu_{\rm e}\|\sym(\nabla u-P)\|^2+\mu_{\rm c}\|{\rm skew}(\nabla u-P)\|^2+\frac{\lambda_{\rm e}}{2}(\tr(\nabla u-P))^2\\
&+\st\|\sym P\|^2+\frac{\lambda_{\mathrm{micro}}}{2}(\tr P)^2+\frac{\mu_{\rm macro} L_{\rm c}^2}{2}\|\Curl P\|^2\Big)\,\di x\,
\end{align}
is coercive, i.e. there exists a constant $C>0$ such that
\begin{align}
C(\|\nabla u\|^2_{{\rm L}^2(\Omega)}+\|P\|^2_{{\rm H}(\Curl;\Omega)})\leq I(\nabla u,P)
\end{align}
for all $u\in {\rm H}^1_0(\Omega)$ and $P\in{\rm H}_0({\rm Curl}\, ; \Omega)$. 
This coercivity follows even for $\mu_{\rm c}=0$, due to the following result:
\begin{tw}{\rm \cite{NeffPaulyWitsch,NPW2,NPW3,BNPS2}}\ \label{wdn}There exists a positive constant $C$, only depending on $\Omega$, such that for all $P\in{\rm H}_0({\rm Curl}\, ; \Omega)$ the following estimate holds:
	\begin{align*}
\qquad \qquad 	{\| P\|_{{\rm H}(\mathrm{Curl})}^2}:=\| P\|_{{\rm L}^2(\Omega)}^{ {2}}+\| \Curl P\|_{{\rm L}^2(\Omega)}^{ {2}}&\leq C\,(\| {\rm sym} P\|^2_{{\rm L}^2(\Omega)}+\| \Curl P\|^2_{{\rm L}^2(\Omega)}). \qquad \qquad \qquad \qquad \quad\blacksquare
	\end{align*}
\end{tw}
The system \eqref{1.2} is considered with the boundary conditions 
\begin{equation}
\begin{split}
u(x,t)=g(x,t)\quad \mathrm{and\,the\, non-homogeneous\,tangential\,condition}\quad P_i(x,t)\times n(x)=G_i(x,t)
\end{split}
\label{1.3}
\end{equation}
for $(x,t)\in \partial\Omega\times [0,T]$, where $n$ is the unit normal vector at the surface $\partial\Omega$, $\times$ denotes the vector product and $P_i$ ($i=1,2,3$) are the rows of $P$. The model is also driven by the following initial conditions 
\begin{equation}
\begin{split}
u(x,0)=u^{(0)}(x)\,,\quad u_{,t}(x,0)=u^{(1)}(x)\,,\quad P(x,0)=P^{(0)}(x)\,,\quad P_{,t}(x,0)=P^{(1)}(x)
\end{split}
\label{1.4}
\end{equation}
for $x\in\Omega$.
\begin{de}
	We say that the initial data $(u^{(0)},u^{(1)},P^{(0)},P^{(1)})$ satisfy the compatibility condition if
	\begin{equation}
	\begin{split}
	u^{(0)}(x)&=g(x,0)\,,\qquad\ \,
	u^{(1)}(x)=g_{,t}(x,0)\,,\\[1ex]
	P^{(0)}_{i}(x)&=G_i(x,0)\,,\qquad
	P^{(1)}_{i}(x)=(G_{i})_{,t}(x,0)
	\end{split}
	\label{comcon}
	\end{equation}
	for $x\in\partial\Omega$ and $i=1,2,3$, where $(G_i)_{,t}$ denotes the time derivative of the function $G_i$.
\end{de}
\noindent The main result of \cite{GhibaNeffExistence} is:
\begin{tw} {\rm (Existence of solution with homogeneous boundary conditions)} 
	Let us assume that the constitutive parameters satisfy  \eqref{condpara} and the initial data are such that  
	\begin{equation}
	(u^{(0)}, u^{(1)}, P^{(0)}, P^{(1)})\in {\rm H}^1_0(\Omega )\times {\rm H}^1_0(\Omega )\times {\rm H}_0(\Curl;\Omega )\times {\rm H}_0(\Curl;\Omega )\,.
	\label{1.6}
	\end{equation}
	Additionally, assume that
$
	\Dyw\big(2\mu_{\rm e}\sym(\nabla u^{(0)}-P^{(0)})+2\mu_{\rm c}\skyw(\nabla u^{(0)}-P^{(0)})+\lambda_{\rm e}\tr(\nabla u^{(0)}-P^{(0)})\id\big)\in {\rm L}^2(\Omega)\,,
	$
$
	\Curl\Curl P^{(0)}\in {\rm L}^2(\Omega)
$
	and
$
	F\in {\rm C}^1([0,T);{\rm L}^2(\Omega))\,,\  M\in {\rm C}^1([0,T);{\rm L}^2(\Omega)).
$
	Then, the system \eqref{1.2} with homogeneous boundary conditions \eqref{1.3} and initial conditions \eqref{1.4} possesses a global in time, unique solution $(u, P)$ with the regularity: for all times $T>0$ 
	\begin{equation}
	\begin{split}
	u\in {\rm C}^1([0,T);{\rm H}^1_0(\Omega ))\,,&\quad u_{,tt}\in {\rm C}((0,T);{\rm L}^2(\Omega ))\,,\\[1ex]
	P\in {\rm C}^1([0,T); {\rm H}_0(\Curl;\Omega ))\quad &\mathrm{and}\quad P_{,tt}\in {\rm C}((0,T);{\rm L}^2(\Omega  ))\,.
	\end{split}
	\label{1.9}
	\end{equation}
	Moreover,
$
	\Dyw\big(2\mu_{\rm e}\sym(\nabla u-P)+2\mu_{\rm c}\skyw(\nabla u-P)+\lambda_{\rm e}\tr(\nabla u-P)\id\big)\in {\rm C}((0,T);{\rm L}^2(\Omega  ))$
	and 
$
	\Curl\Curl P\in {\rm C}((0,T);{\rm L}^2(\Omega  ))\,.
$
	\label{existhomo}
\end{tw}

\section{Existence of solution with non-homogeneous boundary conditions in dynamics}
\renewcommand{\theequation}{\thesection.\arabic{equation}}
\setcounter{equation}{0}%
	In order to prove existence and uniqueness of solution of the system \eqref{1.2} with non-homogeneous boundary conditions we will rewrite the tangential condition on the micro-distortion tensor and  use the result from the paper \cite{GhibaNeffExistence} together with  the new property of the extension operator presented in the previous section. 
\begin{tw} {\rm (Existence of solution with non-homogeneous boundary conditions)} 
	Let us assume that the constitutive parameters satisfy  \eqref{condpara} and the initial data are such that
\begin{equation}
(u^{(0)}, u^{(1)}, P^{(0)}, P^{(1)})\in {\rm H}^1(\Omega )\times {\rm H}^1(\Omega )\times {\rm H}(\Curl;\Omega )\times {\rm H}(\Curl;\Omega )\,
\label{2.27}
\end{equation}
and that the compatibility condition \eqref{comcon} holds.
Additionally, assume that 
\begin{equation}
\Dyw\big(2\mu_{\rm e}\sym(\nabla u^{(0)}-P^{(0)})+2\mu_{\rm c}\skyw(\nabla u^{(0)}-P^{(0)})+\lambda_{\rm e}\tr(\nabla u^{(0)}-P^{(0)})\id\big)\in {\rm L}^2(\Omega)\,,
\label{2.28}
\end{equation} 
\begin{equation}
\Curl\Curl P^{(0)}\in {\rm L}^2(\Omega)
\label{2.29}
\end{equation}
and
\begin{equation}
F\in {\rm C}^1([0,T);{\rm L}^2(\Omega))\,,\quad M\in {\rm C}^1([0,T);{\rm L}^2(\Omega))\,,
\label{forces1}
\end{equation}
\begin{equation}
g\in {\rm C}^3([0,T);{\rm H}^{\frac{3}{2}}(\partial\Omega))\,,\quad G_i\in {\rm C}^3([0,T);\widetilde{\chi}_{\partial\Omega})\quad i=1,2,3\,.
\label{boundary}
\end{equation}
Then, the system \eqref{1.2} with boundary conditions \eqref{1.3} and initial conditions \eqref{1.4} possesses a global in time, unique solution $(u, P)$ with the regularity: for all times $T>0$ 
\begin{equation}
\begin{split}
u\in {\rm C}^1([0,T);{\rm H}^1(\Omega ))\,,&\quad u_{,tt}\in {\rm C}((0,T);{\rm L}^2(\Omega ))\,,\\[1ex]
P\in {\rm C}^1([0,T); {\rm H}(\Curl;\Omega ))\quad &\mathrm{and}\quad P_{,tt}\in {\rm C}((0,T);{\rm L}^2(\Omega))\,.
\end{split}
\label{2.32}
\end{equation}
Moreover,
$
\Dyw\big(2\mu_{\rm e}\sym(\nabla u-P)+2\mu_{\rm c}\skyw(\nabla u-P)+\lambda_{\rm e}\tr(\nabla u-P)\id\big)\in {\rm C}((0,T);{\rm L}^2(\Omega  ))
$
and 
$
\Curl\Curl P\in {\rm C}((0,T);{\rm L}^2(\Omega  ))\,.
$
\end{tw}
\begin{proof} The uniqueness is trivial, while for the existence we  remove the boundary conditions \eqref{1.3}.
The standard extension theorem (see Adams's book, Theorem 7.53 of \cite{adamssobolev}) yields that there exists an extension $\widetilde{g}\in {\rm C}^3([0,T);{\rm H}^2(\Omega))$ of $g$ in $\Omega$. Additionally, from Corollary  \ref{corn} we obtain that there exists  $\widetilde{G}_i\in {\rm C}^3([0,T);{\rm H}(\curl;\Omega))$ ($i=1, 2, 3$) such that $n\times \widetilde{G}_i=G_i$ on $\partial\Omega$ in the sense of ${\rm H}^{-\frac{1}{2}}(\partial\Omega)$, $\curl\curl\widetilde{G}_i=0$ and $\nabla \widetilde{G}_i\in {\rm H}^{1}(\Omega)$.

Note that the property  $\curl\curl\widetilde{G}_i=0$ is not sufficient to prove the existence result, since we also need to know that ${\rm Div} (\sym P), {\rm Div}({\rm skew} P)\in {\rm L}^2(\Omega)$, which are unknown informations from the construction given by Alonso and Valli \cite{tantrace}.

We set $u=\bar{u}-\widetilde{g}$ and $P=\bar{P}-\widetilde{G}$, where 
$
\widetilde{G}=\left(
\widetilde{G}_1^T\,,\,
\widetilde{G}_2^T\,,\,
\widetilde{G}_3^T
\right)^T\!\!.
$
Observe that, in order to find the solution $(\bar{u},\bar{P})$ of the problem \eqref{1.2}-\eqref{1.4}, we have to find a solution $(u,P)$ of the following system
\begin{equation}
\begin{split}
u_{,tt}=&\Dyw\big(2\mu_{\rm e}\sym(\nabla u-P)+2\mu_{\rm c}\skyw(\nabla u-P)+\lambda_{\rm e}\tr(\nabla u-P)\id\big)+\widetilde{F}\,,\\[1ex]
P_{,tt}=&\,2\,\mu_{\rm e}\sym(\nabla u-P)+2\mu_{\rm c}\skyw(\nabla u-P)+\lambda_{\rm e}\tr(\nabla u-P)\id\\[1ex]
&-(2\st\sym P+\lambda_{\mathrm{micro}}(\tr P)\id)-\mu_{\rm macro} L_{\rm c}^2\Curl\Curl P\,+\widetilde{M}\,,
\end{split}
\label{3.1}
\end{equation}
where
\begin{equation}
\begin{split}
\widetilde{F}=F-\widetilde{g}_{,tt}+\Dyw\big(2\mu_{\rm e}\sym(\nabla\widetilde{g}-\widetilde{G})+2\mu_{\rm c}\skyw(\nabla\widetilde{g}-\widetilde{G})+\lambda_{\rm e}\tr(\nabla\widetilde{g}-\widetilde{G})\id\big)
\end{split}
\label{3.2}
\end{equation}
and
\begin{equation}
\begin{split}
\widetilde{M}=&\,M-\widetilde{G}_{,tt}+2\mu_{\rm e}\sym(\nabla\widetilde{g}-\widetilde{G})+2\mu_{\rm c}\skyw(\nabla\widetilde{g}-\widetilde{G})+\lambda_{\rm e}\tr(\nabla\widetilde{g}-\widetilde{G})\id\\[1ex]
&-(2\st\sym\widetilde{G}+\lambda_{\mathrm{micro}}(\tr\widetilde{G})\id)\,.
\end{split}
\label{3.3}
\end{equation}
Now, the system \eqref{3.1} will be considered with the boundary conditions 
\begin{equation}
\begin{split}
u(x,t)=0\quad \mathrm{and\,the\,homogeneous\,tangential\,condition}\quad P_i(x,t)\times n(x)=0
\end{split}
\label{3.4}
\end{equation}
for $(x,t)\in \partial\Omega\times (0,T)$ and initial conditions
\begin{equation}
\begin{split}
u(x,0)&=\bar{u}^{(0)}(x)-\widetilde{g}(x,0)=u^{(0)}(x)\,,\ \ \ \  \ \ 
u_{,t}(x,0)=\bar{u}^{(1)}(x)-\widetilde{g}_{,t}(x,0)=u^{(1)}(x)\,,\\[1ex]
P(x,0)&=\bar{P}^{(0)}(x)-\widetilde{G}(x,0)=P^{(0)}(x)\,,\ \ \ \ 
P_{,t}(x,0)=\bar{P}^{(1)}(x)-\widetilde{G}_{,t}(x,0)=P^{(1)}(x)
\end{split}
\label{3.5}
\end{equation}
for $x\in\Omega$.
Using the assumptions on the functions $\widetilde{g}$ and $\widetilde{G}$ we conclude that $\widetilde{F}\in {\rm C}^1([0,T);{\rm L}^2(\Omega))$ and $\widetilde{M}\in {\rm C}^1([0,T);{\rm L}^2(\Omega))$. Hence, existence and uniqueness of solution for the system \eqref{3.1} with initial-boundary conditions \eqref{3.4} and \eqref{3.5} immediately follow from Theorem \ref{existhomo} and the proof is completed.\end{proof}
\begin{uwa}
	\begin{enumerate}\item[]
\item The property  $\curl\curl\widetilde{G}_i=0$ of the extension (which is known from the construction given by Alonso and Valli \cite{tantrace})  is  sufficient  to prove the existence of the weak solution corresponding to non-homogeneous boundary conditions. In this case, we should rewrite the system  \eqref{1.2} with the boundary conditions \eqref{1.3} and initial conditions \eqref{1.4} as an abstract Cauchy problem in the ${\rm H}^{-1}(\Omega)\times [{\rm H}_0(\Curl; \Omega)]^*$ setting, where  $[{\rm H}_0(\Curl; \Omega)]^*$ is the dual of ${\rm H}_0(\Curl; \Omega)$, i.e. for $F,M\in {\rm C}^1([0,T);{\rm H}^{-1}(\Omega))\times  {\rm C}^1([0,T);[{\rm H}_0(\Curl; \Omega)]^*)$. Since we do not need to know a priori  that ${\rm Div} (\sym P), {\rm Div}({\rm skew} P)\in {\rm L}^2(\Omega)$,  the problem being formulated in a weak sense, the assumption $G_i\in {\rm C}^3([0,T);{\chi}_{\partial\Omega})$ leads us to an existence result of the weak solution.
\item However, when we consider the problem of higher regularity, the ${\rm L}^2(\Omega)\times {\rm L}^2(\Omega)$ setting is needed, and therefore the assumption $G_i\in {\rm C}^3([0,T);\widetilde{\chi}_{\partial\Omega})$ suffices to have the existence of the strong solution.
\item Our remark that our analysis covers also the situation when the boundary condition imposed on the micro-distortion $P$ is related to the displacement $u$, i.e. a coupling condition of the type
$
P_i\times n=\nabla u\times n \  \text{on}\  \partial \Omega,
$
since this implies
$
G_i=(\nabla g)_i\times n\ \in {\rm C}^3([0,T);\widetilde{\chi}_{\partial\Omega}),$ $ \  \widetilde{\chi}_{\partial\Omega}= \{v\in {\rm H}^{\frac{1}{2}}(\partial\Omega): \langle v, n\rangle=0\}$ as long as  $g\in {\rm C}^3([0,T);{\rm H}^{\frac{3}{2}}(\partial\Omega)).
$
\end{enumerate}
\end{uwa}
\section{Existence of solution with non-homogeneous boundary conditions in the static case}\renewcommand{\theequation}{\thesection.\arabic{equation}}
\setcounter{equation}{0}%
The equilibrium  equations of  the relaxed  micromorphic material in the static case are given by 
\begin{align}
\label{ecst}
0=&\Dyw\big(2\mu_{\rm e}\sym(\nabla u-P)+2\mu_{\rm c}\skyw(\nabla u-P)+\lambda_{\rm e}\tr(\nabla u-P)\id\big)+F\nn\\[1ex]
0=&\,2\,\mu_{\rm e}\sym(\nabla u-P)+2\mu_{\rm c}\skyw(\nabla u-P)+\lambda_{\rm e}\tr(\nabla u-P)\id\nn\\[1ex]
&-(2\st\sym P+\lambda_{\mathrm{micro}}(\tr P)\id)-\mu_{\rm macro} L_{\rm c}^2\Curl\Curl P+M\,,
\end{align}
in $\Omega$, where $F:\Omega\rightarrow \R^3$  and $M:\Omega\rightarrow \R^{3\times 3}$ are given.

Consistently with our previous development, we consider the weaker (compared  to the classical) boundary conditions
\begin{align} \label{bc}
{u}({x})=g(x) \   \text{and the {non-homogeneous tangential condition}}  \  {P}_i({x})\times\,n(x) =G_i(x) \ \  \textrm{for all} \ \ {x}\in\partial \Omega.
\end{align}

Regarding the existence of a weak solution $(u,P)$ of this problem, we have established \cite{NeffGhibaMadeoLazar} the following result:
\begin{tw}
	Assume that
	\begin{itemize}
		\item[i)]  the constitutive coefficients satisfy   the inequalities \eqref{condpara};
		\item[ii)] $F\in L^2(\Omega)$, $M\in L^2(\Omega)$;
		\item[iii)] the boundary conditions are homogeneous:  $g=0$ and $G_i=0$.
	\end{itemize}
	Then there exists one and only one weak  solution $({u},{P})\in{\rm H}^1_0(\Omega)\times {\rm H}_0(\Curl; \Omega)$ of the problem defined by \eqref{ecst} and \eqref{bc}, i.e. which satisfies
	\begin{align}
\int_{\Omega}\Big(&2\,\mu_{\rm e}\langle\sym(\nabla u-P),\sym(\nabla \overline{u}-\overline{P})\rangle+2\,\mu_{\rm c}\langle{\rm skew}(\nabla u-P),{\rm skew}(\nabla \overline{u}-\overline{P})\rangle+\lambda_{\rm e}\tr(\nabla u-P)\tr(\nabla \overline{u}-\overline{P})\notag\\
&\ \ +2\mu_{\rm micro}\langle\sym P,\sym \overline{P}\rangle+{\lambda_{\mathrm{micro}}}\tr P\,\tr\overline{P}+{\mu_{\rm macro} L_{\rm c}^2}\langle\Curl P,\Curl \overline{P}\rangle\Big)\,\di x\\&=	\int_\Omega (\langle F, \overline{u}\rangle+\langle M, \overline{P}\rangle )\,\di x \ \ \qquad  \textrm{for all}\ \ \ (\overline{u},\overline{P})\in{\rm H}^1_0(\Omega)\times {\rm H}_0(\Curl; \Omega).\notag
	\end{align}
\end{tw}

In the following we give an extension of this result to the case of non-homogeneous boundary conditions.
\begin{tw}
	Assume that
\begin{itemize}
	\item[i)]  the constitutive coefficients satisfy   the inequalities \eqref{condpara};
	\item[ii)] $F\in L^2(\Omega)$, $M\in L^2(\Omega)$;
	\item[iii)] the boundary conditions are such that 	$
	g\in {\rm H}^{\frac{3}{2}}(\partial\Omega)\,,\quad G_i\in\widetilde{\chi}_{\partial\Omega}\quad i=1,2,3\,.
	$
\end{itemize}
Then there exists one and only one weak  solution of the problem defined by \eqref{ecst} and \eqref{bc} of the problem defined by \eqref{ecst} and \eqref{bc}, such that $u=\widetilde{u}+\widetilde{g}$, $P=\widetilde{P}+\widetilde{G}$, where $(\widetilde{u},\widetilde{P})\in{\rm H}^1_0(\Omega)\times {\rm H}_0(\Curl; \Omega)$ is the weak solution of the following problem
\begin{align}
\int_{\Omega}\Big(&2\,\mu_{\rm e}\langle\sym(\nabla \widetilde{u}-\widetilde{P}),\sym(\nabla \overline{u}-\overline{P})\rangle+2\,\mu_{\rm c}\langle{\rm skew}(\nabla \widetilde{u}-\widetilde{P}),{\rm skew}(\nabla \overline{u}-\overline{P})\rangle+\lambda_{\rm e}\tr(\nabla \widetilde{u}-\widetilde{P})\tr(\nabla \overline{u}-\overline{P})\notag\\
&\ \ \ \ \ \ \ +2\mu_{\rm micro}\langle\sym \widetilde{P},\sym \overline{P}\rangle+{\lambda_{\mathrm{micro}}}\tr \widetilde{P}\,\tr\overline{P}+{\mu_{\rm macro} L_{\rm c}^2}\langle\Curl \widetilde{P},\Curl \overline{P}\rangle\Big)\,\di x\\&=	\int_\Omega (\langle \widetilde{F}, \overline{u}\rangle+\langle \widetilde{M}, \overline{P}\rangle )\,\di x \ \ \qquad  \textrm{for all}\ \ \ (\overline{u},\overline{P})\in{\rm H}^1_0(\Omega)\times {\rm H}_0(\Curl; \Omega),\notag 
\end{align}
where
we have set
$
\widetilde{F}=F+\Dyw\big(2\mu_{\rm e}\sym(\nabla\widetilde{g}-\widetilde{G})+2\mu_{\rm c}\skyw(\nabla\widetilde{g}-\widetilde{G})+\lambda_{\rm e}\tr(\nabla\widetilde{g}-\widetilde{G})\id\big)
$
and
$
\widetilde{M}=\,M+2\mu_{\rm e}\sym(\nabla\widetilde{g}-\widetilde{G})+2\mu_{\rm c}\skyw(\nabla\widetilde{g}-\widetilde{G})+\lambda_{\rm e}\tr(\nabla\widetilde{g}-\widetilde{G})\id-(2\st\sym\widetilde{G}+\lambda_{\mathrm{micro}}(\tr\widetilde{G})\id)\,.
$
\end{tw}

\begin{proof}
Using  the assumptions on the functions $\widetilde{g}$ and $\widetilde{G}$ we conclude that $\widetilde{F}\in {\rm L}^2(\Omega)$ and $\widetilde{M}\in {\rm L}^2(\Omega)$ as long as $F\in {\rm L}^2(\Omega)$, $M\in {\rm L}^2(\Omega)$. Hence, the existence and uniqueness of the weak solution solution for non-homogeneous boundary conditions   follows immediately.
\end{proof}

Let us now assume that the functions which define the boundary conditions are such that 	$
g\in {\rm H}^{\frac{3}{2}}(\partial\Omega)\,,\quad G_i\in {\chi}_{\partial\Omega}\quad i=1,2,3\,
$. We recall that from the standard trace theorem and Theorem \ref{lem:2.2} we obtain that there exist the functions $\widetilde{g}\in {\rm H}^2(\Omega)$ and $\widetilde{G}_i\in  {\rm H}(\Curl;\Omega)$ such that $\curl\curl \widetilde{G}_i=0$ and $\widetilde{G}\big|_{\partial \Omega}=g$ and $\widetilde{G}_i\times n=G_i$, $i=1,2,3$, in the sense of ${\rm H}^{-\frac{1}{2}}(\partial\Omega)$. In this case, we may rewrite the problem for $F,M\in {\rm H}^{-1}(\Omega)\times [{\rm H}_0(\Curl; \Omega)]^*$. Therefore, we arrive to a similar existence result, where all the derivative in the definitions of $\widetilde{F}$  and  $\widetilde{M}$ are weak derivates.

However, while $ G_i\in {\chi}_{\partial\Omega}, i=1,2,3\,
$ is sufficient to obtain an existence results of the weak solution for appropriate regularities of the load functions, when the higher-regularity problem is considered we expect that the assumption $ G_i\in \widetilde{\chi}_{\partial\Omega}, i=1,2,3\,
$ will play a crucial role.
\section{Acknowledgement}
The  work the second author  was supported by a grant of the Romanian Ministry of Research
and Innovation, CNCS--UEFISCDI, project number
PN-III-P1-1.1-TE-2019-0397, within PNCDI III.

\bibliographystyle{plain} 

\begin{thebibliography}{10}
	
	\bibitem{adamssobolev}
	R.~Adams and J.F. Fournier.
	\newblock {\em Sobolev {S}paces}, volume 140 of {\em Pure and Applied
		Mathematics (Amsterdam)}.
	\newblock Elsevier/Academic Press, Amsterdam, second edition, 2003.
	
	\bibitem{aivaliotis2019microstructure}
	A.~Aivaliotis, A.~Daouadji, G.~Barbagallo, D.~Tallarico, P.~Neff, and A.~Madeo.
	\newblock Microstructure-related {Stoneley} waves and their effect on the
	scattering properties of a {2D Cauchy}/relaxed-micromorphic interface.
	\newblock {\em Wave Motion}, 90:99--120, 2019.
	
	\bibitem{tantrace}
	A.~Alonso and A.~Valli.
	\newblock Some remarks on the characterization of the space of tangential
	traces of {$H({\rm rot};\Omega)$} and the construction of an extension
	operator.
	\newblock {\em Manuscripta Math.}, 89(2):159--178, 1996.
	
	\bibitem{barbagallo2017transparent}
	G.~Barbagallo, A.~Madeo, M.V. d'Agostino, R.~Abreu, I.D. Ghiba, and P.~Neff.
	\newblock Transparent anisotropy for the relaxed micromorphic model:
	macroscopic consistency conditions and long wave length asymptotics.
	\newblock {\em Int. J. Solids Struct.}, 120:7--30, 2017.
	
	\bibitem{barbagallo2019relaxed}
	G.~Barbagallo, D.~Tallarico, M.V. d'Agostino, A.~Aivaliotis, P.~Neff, and
	A.~Madeo.
	\newblock Relaxed micromorphic model of transient wave propagation in
	anisotropic band-gap metastructures.
	\newblock {\em Int. J. Solids and Struct.}, 162:148--163, 2019.
	
	\bibitem{BNPS2}
	S.~Bauer, P.~Neff, D.~Pauly, and G.~Starke.
	\newblock Dev-{Div} and {DevSym}-{DevCurl} inequalities for incompatible square
	tensor fields with mixed boundary conditions.
	\newblock {\em ESAIM: COCV}, 22(1):112--133, 2016.
	
	\bibitem{tandyw}
	C.~B\'egue, C.~Conca, F.~Murat, and O.~Pironneau.
	\newblock Les \'equations de {S}tokes et de {N}avier-{S}tokes avec des
	conditions aux limites sur la pression.
	\newblock In {\em Nonlinear partial differential equations and their
		applications. {C}oll\`ege de {F}rance {S}eminar, {V}ol.\ {IX} ({P}aris,
		1985--1986)}, volume 181 of {\em Pitman Res. Notes Math. Ser.}, pages
	179--264. Longman Sci. Tech., Harlow, 1988.
	
	\bibitem{electrobook}
	M.~Cessenat.
	\newblock {\em Mathematical {M}ethods in {E}lectromagnetism}, volume~41 of {\em
		Series on Advances in Mathematics for Applied Sciences}.
	\newblock World Scientific Publishing Co., Inc., River Edge, NJ, 1996.
	\newblock Linear theory and applications.
	
	\bibitem{d2017effective}
	M.V. d'Agostino, G.~Barbagallo, I.D. Ghiba, B.~Eidel, P.~Neff, and A.~Madeo.
	\newblock Effective description of anisotropic wave dispersion in mechanical
	band-gap metamaterials via the relaxed micromorphic model.
	\newblock {\em Journal of Elasticity}, accepted,
	https://doi.org/10.1007/s10659-019-09753-9, 2019.
	
	\bibitem{d2017panorama}
	M.V. d'Agostino, G.~Barbagallo, I.D. Ghiba, A.~Madeo, and P.~Neff.
	\newblock A panorama of dispersion curves for the weighted isotropic relaxed
	micromorphic model.
	\newblock {\em Z. Angew. Math. Mech.}, 97(11):1436--1481, 2017.
	
	\bibitem{Eringen99}
	A.C. Eringen.
	\newblock {\em Microcontinuum {F}ield {T}heories.}
	\newblock Springer, Heidelberg, 1999.
	
	\bibitem{Eringen64}
	A.C. Eringen and E.S. Suhubi.
	\newblock Nonlinear theory of simple micro-elastic solids. {I}.
	\newblock {\em Int. J. Eng. Sci.}, 2:189--203, 1964.
	
	\bibitem{blanco2000large}
	A.~Blanco et~al.
	\newblock Large-scale synthesis of a silicon photonic crystal with a complete
	three-dimensional bandgap near 1.5 micrometres.
	\newblock {\em Nature}, 405(6785):437--440, 2000.
	
	\bibitem{liu2000locally}
	Z.~Liu et~al.
	\newblock Locally resonant sonic materials.
	\newblock {\em Science}, 289(5485):1734--1736, 2000.
	
	\bibitem{temharmonic}
	C.~Foias and R.~Temam.
	\newblock Remarques sur les \'equations de {N}avier-{S}tokes stationnaires et
	les ph\'enom\`enes successifs de bifurcation.
	\newblock {\em Ann. Scuola Norm. Sup. Pisa Cl. Sci. (4)}, 5(1):28--63, 1978.
	
	\bibitem{tantracena1}
	V.~Georgescu.
	\newblock Some boundary value problems for differential forms on compact
	{R}iemannian manifolds.
	\newblock {\em Ann. Mat. Pura Appl. (4)}, 122:159--198, 1979.
	
	\bibitem{GhibaNeffExistence}
	I.~D. Ghiba, P.~Neff, A.~Madeo, L.~Placidi, and G.~Rosi.
	\newblock The relaxed linear micromorphic continuum: Existence, uniqueness and
	continuous dependence in dynamics.
	\newblock {\em Math. Mech. Solids}, 20:1171--1197, 2015.
	
	\bibitem{Giraultbook}
	V.~Girault and P.-A. Raviart.
	\newblock {\em Finite {E}lement {M}ethods for {N}avier-{S}tokes {E}quations},
	volume~5 of {\em Springer Series in Computational Mathematics}.
	\newblock Springer-Verlag, Berlin, 1986.
	\newblock Theory and algorithms.
	
	\bibitem{hughes1978classical}
	Th. Hughes, J.E., and Marsden.
	\newblock Classical elastodynamics as a linear symmetric hyperbolic system.
	\newblock {\em Journal of Elasticity}, 8(1):97--110, 1978.
	
	\bibitem{lasiecka1986non}
	I.~Lasiecka, J.L. Lions, and R.~Triggiani.
	\newblock Non homogeneous boundary value problems for second order hyperbolic
	operators.
	\newblock {\em J. Math. {P}ures Appl.}, 65(2):149--192, 1986.
	
	\bibitem{madeo2019dispersion}
	A.~Madeo and P.~Neff.
	\newblock Dispersion of waves in micromorphic media and metamaterials.
	\newblock In {\em Handbook of Nonlocal Continuum Mechanics for Materials and
		Structures}, pages 713--739. Springer, 2019.
	
	\bibitem{madeo2017review}
	A.~Madeo, P.~Neff, G.~Barbagallo, M.V. d'Agostino, and I.D. Ghiba.
	\newblock A review on wave propagation modeling in band-gap metamaterials via
	enriched continuum models.
	\newblock In {\em Mathematical Modelling in Solid Mechanics}, pages 89--105.
	Springer, 2017.
	
	\bibitem{madeo2016complete}
	A.~Madeo, P.~Neff, M.V. d'Agostino, and G.~Barbagallo.
	\newblock Complete band gaps including non-local effects occur only in the
	relaxed micromorphic model.
	\newblock {\em Comptes Rendus M{\'e}canique}, 344(11-12):784--796, 2016.
	
	\bibitem{MadeoNeffGhibaW}
	A.~Madeo, P.~Neff, I.~D. Ghiba, L.~Placidi, and G.~Rosi.
	\newblock Wave propagation in relaxed linear micromorphic continua: modelling
	metamaterials with frequency band-gaps.
	\newblock {\em Cont. Mech. Therm.}, 27:551--570, 2015.
	
	\bibitem{madeo2016reflection}
	A.~Madeo, P.~Neff, I.D. Ghiba, and G.~Rosi.
	\newblock Reflection and transmission of elastic waves in non-local band-gap
	metamaterials: a comprehensive study via the relaxed micromorphic model.
	\newblock {\em J. Mech. Phys. Solids}, 95:441--479, 2016.
	
	\bibitem{Mindlin64}
	R.D. Mindlin.
	\newblock Micro-structure in linear elasticity.
	\newblock {\em Arch. Rat. Mech. Anal.}, 16:51--77, 1964.
	
	\bibitem{NeffMadeoEidel2019JELast}
	P.~Neff, B.~Eidel, M.V. d'Agostino, and A.~Madeo.
	\newblock Identification of scale-independent material parameters in the
	relaxed micromorphic model through model-adapted first order homogenization.
	\newblock {\em Journal of Elasticity}, accepted,
	https://doi.org/10.1007/s10659-019-09752-w, 2019.
	
	\bibitem{NeffGhibaMadeoLazar}
	P.~Neff, I.~D. Ghiba, M.~Lazar, and A.~Madeo.
	\newblock The relaxed linear micromorphic continuum: well-posedness of the
	static problem and relations to the gauge theory of dislocations.
	\newblock {\em Q. J. Mech. Appl. Math.}, 68:53--84, 2015.
	
	\bibitem{NeffGhibaMicroModel}
	P.~Neff, I.~D. Ghiba, A.~Madeo, L.~Placidi, and G.~Rosi.
	\newblock A unifying perspective: the relaxed linear micromorphic continuum.
	\newblock {\em Cont. Mech. Therm.}, 26:639--681, 2014.
	
	\bibitem{NPW2}
	P.~Neff, D.~Pauly, and K.J. Witsch.
	\newblock A canonical extension of {K}orn's first inequality to {${\rm
			H(Curl)}$} motivated by gradient plasticity with plastic spin.
	\newblock {\em C. R. Acad. Sci. Paris, Ser. I}, 349:1251--1254, 2011.
	
	\bibitem{NPW3}
	P.~Neff, D.~Pauly, and K.J. Witsch.
	\newblock Maxwell meets {K}orn: a new coercive inequality for tensor fields in
	{$\mathbb{R}^{N\times N}$} with square-integrable exterior derivative.
	\newblock {\em Math. Methods Appl. Sci.}, 35:65--71, 2012.
	
	\bibitem{NeffPaulyWitsch}
	P.~Neff, D.~Pauly, and K.J. Witsch.
	\newblock Poincar\'{e} meets {K}orn via {M}axwell: Extending {K}orn's first
	inequality to incompatible tensor fields.
	\newblock {\em J. Differential Equations}, 258:1267--1302, 2015.
	
	\bibitem{tantracena2}
	L.~Paquet.
	\newblock Probl\`emes mixtes pour le syst\`eme de {M}axwell.
	\newblock {\em C. R. Acad. Sci. Paris S\'er. A-B}, 289(3):A191--A194, 1979.
	
	\bibitem{tantracena3}
	L.~Paquet.
	\newblock Probl\`emes mixtes pour le syst\`eme de {M}axwell.
	\newblock {\em Ann. Fac. Sci. Toulouse Math. (5)}, 4(2):103--141, 1982.
	
	\bibitem{picardharmonic}
	R.~Picard.
	\newblock On the boundary value problems of electro- and magnetostatics.
	\newblock {\em Proc. Roy. Soc. Edinburgh Sect. A}, 92(1-2):165--174, 1982.
	
	\bibitem{valent2013boundary}
	T.~Valent.
	\newblock {\em Boundary {V}alue {P}roblems of {F}inite {E}lasticity: {L}ocal
		{T}heorems on {E}xistence, {U}niqueness, and {A}nalytic {D}ependence on
		{D}ata}, volume~31.
	\newblock Springer Science \& Business Media, 2013.
	
\end{thebibliography}

\addcontentsline{toc}{section}{References}

\def\cprime{$'$}

\appendix

\section{Linear elastic theory}\setcounter{equation}{0}
In order to put our previous result into proper perspective, we recapitulate in this appendix the well-known stepts for the linear elastic case. The partial differential equations associated to linear elastodynamics are
\begin{equation}
\label{1.2a}
u_{,tt}=\Dyw\big(2\mu_{\rm e}\sym(\nabla u)+\lambda_{\rm e}\tr(\nabla u)\id\big)+F,
\end{equation}
in $\Omega\times (0,T)$, where $F:\Omega\times (0,T)\rightarrow \R^3$ is a given body force. 
We adjoin to \eqref{1.2a} the boundary conditions 
\begin{equation}
\begin{split}
u(x,t)=g(x,t)\quad \mathrm{(only\ \  Dirichlet)}
\end{split}
\label{1.3a}
\end{equation}
for $(x,t)\in \partial\Omega\times [0,T]$ 
and the  following initial conditions 
$
u(x,0)=u^{(0)}(x)\,,\ u_{,t}(x,0)=u^{(1)}(x),
$
for $x\in\Omega$.
	We say that the initial data $(u^{(0)},u^{(1)})$ satisfy the compatibility condition if
	$
	u^{(0)}(x)=g(x,0)\,,\ 
	u^{(1)}(x)=g_{,t}(x,0)\,,
	$
	for $x\in\partial\Omega$.

It is known that, if $f\in {\rm L}^2(\Omega)$  for      homogeneous boundary conditions then there exists a unique solution $u$ having the regularity $$	u\in {\rm C}^1([0,T);{\rm H}^1_0(\Omega )), \qquad u_{,tt}\in C((0,T);{\rm L}^2(\Omega )) .$$

Now, considering non-homogeneous boundary conditions we are looking for a solution 
$$	u\in {\rm C}^1([0,T);{\rm H}^1(\Omega )), \qquad u_{,tt}\in C((0,T);{\rm L}^2(\Omega )),$$
such that 
$u=g$ on $\partial \Omega.$

The boundary condition is well defined if $g$ belongs to ${\rm H}^{1/2}(\partial\Omega)$, since we know that for each $u\in {\rm H}^1(\Omega )$ we may speak about its trace on the boundary which belongs to ${\rm H}^{1/2}(\partial\Omega)$. Moreover, we know that this trace operator is surjective in the sense that for each ${\rm H}^{1/2}(\partial\Omega)$ there exists an extension $\widetilde{g}\in {\rm H}^1(\Omega )$ such that $\widetilde{g}\big|_{\partial \Omega}=g$ in the trace sense.

Therefore, in principle the solution of the inhomogeneous problem is written as $u=\widetilde{u}+\widetilde{g}$, where $\widetilde{g}$ is the extension of $g\in {\rm H}^{3/2}(\partial\Omega)\subset {\rm H}^{1/2}(\partial\Omega)$ to $ {\rm H}^2(\Omega )$ (see the existence result by Lasiecka, Lions and Triggiani \cite[Theorem 2.2]{lasiecka1986non} obtained for general linear hyperbolic equations and also Theorem 1.4 from \cite{hughes1978classical}) and $\widetilde{u}\in {\rm H}_0^1(\Omega)$. 

To be more precise, $\widetilde{u}\in H_0^1(\Omega)$ is solution of the following equation
\begin{eqnarray}
\label{1.2aa}
\widetilde{u}_{,tt}&=&\Dyw\big(2\mu_{\rm e}\sym(\nabla \widetilde{u})+\lambda_{\rm e}\tr(\nabla \widetilde{u})\id\big)+\underbrace{F-\widetilde{g}_{,tt}+\Dyw\big(2\mu_{\rm e}\sym(\nabla \widetilde{g})+\lambda_{\rm e}\tr(\nabla \widetilde{g})\id\big)\nn}_{\textrm{the new right hand side which must belong to}\ \ {\rm L}^2(\Omega)},
\end{eqnarray}
and
$
\widetilde{u}\big|_{\partial \Omega}=u\big|_{\partial \Omega}-\widetilde{g}\big|_{\partial \Omega}=g-g=0.
$
Why do we need $g\in {\rm H}^{3/2}(\partial\Omega)$? Since when we are looking for a solution in the form $u=\widetilde{u}+\widetilde{g}$, with $\widetilde{u}\in {\rm H}_0^1(\Omega)$ a solution of the PDE, we have to apply the differential operator from the right hand side to $\widetilde{g}$ and the obtained result has to belong to ${\rm L}^2(\Omega)$. This is possible, since where exists an extension of $g\in {\rm H}^{3/2}(\partial\Omega)$ which is in ${\rm H}^2(\Omega)$.

Therefore, using the trace operator and the extension to ${\rm H}^1(\Omega)$, for classical elasticity, the case of non-homogeneous boundary conditions may be reduced to homogeneous boundary conditions.

\end{document}